\documentclass[12pt]{amsart}

\newtheorem{theorem}{Theorem}[section]

\theoremstyle{definition}

\theoremstyle{remark}
\newtheorem{remark}[theorem]{Remark}

\numberwithin{equation}{section}

\newcommand{\be}{\begin{equation}}
\newcommand{\ee}{\end{equation}}
\newcommand{\F}{\mathbb{F}}
\newcommand{\Fp}{\mathbb{F}_p}
\newcommand{\Z}{\mathbb{Z}}
\newcommand{\EFp}{E(\Fp)}

\begin{document}

\title{Primality tests for $2^kn-1$ using elliptic curves}


\author{Yu Tsumura}
\address{Department of Mathematics, Purdue University
150 N. University Street, West Lafayette, IN 47907-2067
}

\email{ytsumura@math.purdue.edu}
\thanks{}


\date{}

\dedicatory{}

\begin{abstract}
We propose some primality tests for $2^kn-1$, where $k$, $n \in \Z$, $ k\geq 2$ and $n$ odd.
There are several tests depending on how big $n$ is.
These tests are proved using properties of elliptic curves.
Essentially, the new primality tests are the elliptic curve version of the Lucas-Lehmer-Riesel primality test.

\end{abstract}

\maketitle


\section{Note}
An anonymous referee suggested that Benedict H. Gross already proved the same result about a primality test for Mersenne primes using elliptic curve in \cite{Gross}.

\section{Introduction.}
There are mainly two types of primality tests.
One of them applies to any integer and the other applies only to a special form of integer.
Usually the latter is faster than the former because of its additional information.
Among them, the Lucas-Lehmer primality test for Mersenne numbers $M_k=2^k-1$ is very fast.
The test uses a sequence $S_i$ defined by $S_0=4$ and $S_{i+1}=S_i^2-2$ for $i\geq1$.
The primality test is that $M_k$ is prime if and only if $M_k$ divides $S_{k-2}$.
For a proof, see for example \cite{Crandall-Pomerance}. Also see \cite{Brillhart-Lehmer-Selfridge} and \cite{Williams} for applications of the Lucas sequence for other primality tests.
There is also a generalization of this test called the Lucas-Lehmer-Riesel test which applies to integers of the form $2^kn-1$ with $n<2^k$ (see \cite{Riesel} and \cite{Williams}). 
This test also uses the sequence $S_i$ defined by the above recursion but with a different initial value $S_0$ depending $k$ and $n$.

In this paper we give several primality tests for integers of the form $2^kn-1$ using elliptic curves.
When $n$ is relatively small as in the Lucas-Lehmer-Riesel test, the primality test can be regarded as an analogue of the Lucas-Lehmer-Riesel test.
The new test also uses a sequence defined by recursion. 
For the initial value, we need to take a proper elliptic curve and a point on it.
This corresponds to  the choice of an initial value in the Lucas-Lehmer-Reisel test.
However, when the new test applies to Mersenne numbers $2^k-1$, there exist an  elliptic curve and a point on it which are independent of $k$.
 
Now let us define the sequence.
Let  $p \equiv 3 \pmod 4$ be a prime number and let $E$ be  an elliptic curve defined by $y^2=x^3-mx$ for some integer $m \not \equiv 0 \pmod p$.
Fix a point $Q=(x,y)\in E(\Fp)$ and denote $2^iQ=(x_i,y_i)$ for $i\geq 0$.
On this curve, multiplication of a point by $2$ is described as
\begin{equation}\label{equ:multi}
2(x,y)=\left(\frac{x^4+2mx^2+m^2}{4(x^3-mx)},yR(x) \right)=\left( \left(\frac{x^2+m}{2y}\right)^2, yR(x) \right)
\end{equation}
for some rational function $R(x)$.
(See Example 2.5, page 52 in \cite{Washington}).
Let us define a sequence $S_i$.
Let $S_0=x$ and $S_i=4(x_{i-1}^3-mx_{i-1})$ for $i\geq 1$, that is, $S_i$ is the denominator of $2^iQ$ when $i\geq 1$. 
Alternatively, we could omit a constant  $4$ in the definition of $S_i$.
We refer to this sequence as the sequence $S_i$ with the initial value $Q=(x,y)$, or with the initial value $x$.
Note that $S_i$ depends only on $x$ and $i$.
($S_i$ also depends on $m$. However, it will be clear from the context which $m$ is used.)

\section{Group structure of $E(\Fp)$.}

First, we analyze the structure of the group $E(\Fp)$, where $E$ is an elliptic curve defined by $y^2=x^3-mx$ for some integer $m \not \equiv 0 \pmod p$ and $p \equiv 3 \pmod 4$ is a prime number.
Assume $p+1=2^kn$, where $k\in \Z$, $k\geq 2$ and $n$ is an odd integer.

\begin{theorem}\label{thm:p+1}
In this context, $\#E(\Fp)=p+1$.
\end{theorem}
\begin{proof} See  Theorem 4.23, page 115 in  \cite{Washington}.
\end{proof}

\begin{theorem}\label{thm:group-structure}
In this context,
\[E(\Fp) \cong    \Z_{2^kn}   \textrm{ or } \Z_{2} \oplus \Z_{2^{k-1}n}   \]
 depending on whether $m$ is a non-quadratic residue or a quadratic residue modulo $p$.

\end{theorem}
\begin{proof}

By Theorem \ref{thm:p+1}, we have $\#E(\Fp)=p+1=2^kn$. 
Hence $\EFp \cong \Z_{2^{\alpha}n_1} \oplus \Z_{2^{\beta}n_2}$ for some $n_1$, $n_2$, $\alpha$, $\beta \in \Z $ with $\alpha \leq \beta$ and $\alpha + \beta =k$ and $n_1|n_2$ and $n_1n_2=n$. 
However, in general, $2^{\alpha} n_1$ must divide $ p-1$ by the group structure of elliptic curves. 
(See    Theorem 4.3 and 4.4, page 98 in \cite{Washington}.)
Note that $\gcd(\#E(\Fp),p-1)=\gcd(p+1,p-1)=2$. Therefore $n_1=1$ and $n_2=n$.

If $m$ is a quadratic non-residue (with Jacobi notation, $\left(  \frac{m}{p}\right)=-1$), then only one root of $x^3-mx$ is in $\Fp$.
Hence $E[2] \not \subset \EFp$.
Therefore $\alpha =0$ and $\EFp \cong  \Z_{2^kn}$.

If $\left(  \frac{m}{p}\right)=1$, then $\sqrt{m} \in \Fp$.
Hence all the roots of $x^3-mx$ are in $\Fp$.
Hence $E[2] \subset \EFp$.
So $\alpha \geq1$.
Since $p-1 \equiv 2 \pmod p$ and $2^{\alpha} | p-1$, we have $\alpha=1$.
Therefore $E(\Fp) \cong \Z_{2} \oplus \Z_{2^{k-1}n}$.

\end{proof}

The next theorem is essential to choose an initial value.

\begin{theorem}\label{thm:main}
Let $p\equiv 3 \pmod 4$ be prime and let $E$ be  an elliptic curve defined by $y^2=x^3-mx$ for some integer $m$. 
Assume $p+1=2^kn$, where $k\in \Z$, $k\geq 2$ and $n$ is an odd integer.
Suppose $m$ is a quadratic non-residue modulo $p$.
If $Q=(x,y)\in E$ and $x$ is a quadratic non-residue, then $Q$ has order divisible by $2^k$ in the cyclic group $E(\Fp)\cong \Z_{2^kn}$.
\end{theorem}
\begin{proof}
Since $m$ is a quadratic non-residue, $E(\Fp)\cong\Z_{2^kn}$ by Theorem \ref{thm:group-structure}.
Hence $E(\Fp)$ is cyclic. Let $G$ be a generator of this group and let $tG=Q$ for an integer $t$.

We show that $Q=(x,y) \notin 2E(\Fp)$.
Suppose $(x,y)=2(x_0,y_0)$ for a point $(x_0,y_0) \in E(\Fp)$.
Then by equation \ref{equ:multi}, we have $x=((x_0^2+m)/(2y))^2$.
Hence $x$ is a square in $\Fp$, which contradicts the assumption that $x$ is a quadratic non-residue modulo $p$.
So $Q=(x,y) \notin 2E(\Fp)$.

Therefore $t$ is odd and then $Q$ has order divisible by $2^k$.

\end{proof}

\section{Primality test for $p=2^kn-1$ when $n$ is small.}

Using Theorem \ref{thm:main}, we give  primality tests for integers of the form $p=2^kn-1$, where $k$, $n \in \Z$, $k \geq 2$ and $n$ is an odd integer.
There are two primality tests.
We distinguish them by the relative size of $n$ when compared with $2^k$.
First, let us discuss the case when $n$ is relatively small.

\begin{theorem}\label{thm:algo}
Fix $\lambda>1$. 
Suppose $p=2^kn-1$ with $k\geq 2$ and  an odd integer $n \leq \sqrt{p}/\lambda$.
Assume $p$ is not so small.
More precisely, assume $p$ satisfies $\lambda \sqrt{p} > (p^{1/4}+1)^2$.
Let $E$ be a curve defined by $y^2=x^3-mx$, where $m$ is a quadratic non-residue module $p$.
Then $p$ is prime if and only if 
there exists a point $ Q=(x,y)$ on  $E$ such that
\[\gcd(S_i,p)=1\]
for $i=1$, $2$, $\ldots$, $k-1$ and 
\[S_k\equiv0 \pmod p, \]
where   $S_i$ is a sequence with the initial value $S_0=x$.
\end{theorem}
\begin{proof}
Suppose $p$ is prime. 
Then by Theorem \ref{thm:group-structure}, $E(\Fp) \cong  \Z_{2^kn}$.
Then $\EFp$ has a point $Q=(x,y)$ of order $2^k$. 
Hence $S_i$, with the initial value this $x$, satisfies the conditions of the theorem since $S_i$ is the denominator of $2^iQ$.

Conversely, suppose there exists $Q$ which satisfies the conditions. 
Assume $p$ is composite and let $r$ be a prime divisor such that $r \leq \sqrt{p}$.
Then we have $\gcd(S_i,r)=1$ for $i=1$, $2$, $\ldots$, $k-1$ and $S_k\equiv0 \pmod r$.
Hence in the reduction $E(\F_r)$, $Q$ has an order $\geq 2^k$.
Using the condition on $n$, we have
\[ \lambda \sqrt{p} \leq p/n <2^k \leq \#E(\F_r) \leq (\sqrt{r}+1)^2 \leq (p^{1/4}+1)^2\]
Here the third inequality is by Hasse's Theorem.
However, we assumed that this does not happen. 
Therefore $p$ is prime.

\end{proof}

To make Theorem \ref{thm:algo} into a primality test, we need to find a point $Q$ in the theorem.
To this end we use Theorem \ref{thm:main}.
Let us first state the algorithm.

\textbf{Algorithm.}
Let $p$ be an integer of the form $p=2^kn-1$ with $k\geq2$ and $p$, $n$ satisfy the conditions of Theorem \ref{thm:algo}.
To check whether $p$ is prime, do the following steps.

\begin{enumerate}
	\item Take $x \in \Z$ such that  $\left(  \frac{x}{p}\right)=-1$ and find $y$ such that 
         	 $\left(\frac{x^3-y^2}{p} \right)=1$.
	        Let $m=(x^3-y^2)/x$ mod $p$.
         	Then $Q'=(x,y)$ lies on the curve $E: y^2=x^3-mx$, where $m \not \equiv 0 \pmod p $.
	         The following calculation is done in $E(\Z_p)$.
        	Let $Q=nQ'$. 
        	If $Q=\infty$, then $p$ is composite. 
	        If not, go to Step 2.
	\item Let $S_i$ be the sequence with the inital value $Q$. 
	      Calculate $S_i$ for $i=1$, $2$, $\ldots$, $k-1$ .
	      If $\gcd(S_i,p)>1$ for some $i$, $1\leq i \leq k-1$, then $p$ is composite.
	      If $\gcd(S_i,p)=1$ for $i=1$, $2$, $\ldots$, $k-1$, then go to Step 3.
	\item If $S_k\equiv0 \pmod p$, then $p$ is prime.
      	If not, $p$ is composite.

\end{enumerate}

Let us check why this algorithm works.
In Step 1, we find an elliptic curve $E:y^2=x^3-mx$ and a point $Q$ on $E$ whose $x$-coordinate is a quadratic non-residue.
We have $\left(\frac{m}{p} \right)= \left( \frac{(x^3-y^2)/x}{p} \right)= \left(\frac{x}{p}\right)\left(\frac{x^3-y^2}{p}\right)= -1\cdot1=-1$.
Hence if $p$ is prime, then $Q'$ has order divisible by $2^k$ by Theorem \ref{thm:main}. 
So the order of $Q'$ is $2^kd$, where $d|n$.
Hence $Q=nQ'$ has order $2^k$. Therefore if Step 1 concludes that $p$ is composite, then $p$ is really composite.
Step 2 and Step 3 check  if $Q$ has order $2^k$.
So if Step 2 or Step 3 concludes that $p$ is composite, then $p$ is really composite.
If the algorithm concludes $p$ is prime, then $S_1$ satisfies the conditions of Theorem \ref{thm:algo}.
Therefore $p$ is really prime.

\begin{remark}
Since we know both coordinates of $Q$, we can calculate $nQ$ quickly.
\end{remark}
\begin{remark}
Suppose this test concludes that $p$ is composite because $\gcd(S_i, p)>1$ for some $i$, $1\leq i \leq k-1$ in Step 2.
Then $\gcd(S_i, p)$ might be a proper divisor of $p$ though it might be $p$ itself.
This is the basic idea of the primality testing using elliptic curves proposed by Goldwasser and Kilian (see \cite{Goldwasser-Kilian}).
\end{remark}

\section{Primality test for Mersenne numbers.}

Let us apply the above algorithm for Mersenne numbers $M_k=2^k-1$.
That is, we take $n=1$ and suppose $k\geq3$.
 In this case we do not have to choose the initial value and the elliptic curve as in Step 1.
Note that since $n=1$, the algorithm contains no elliptic curve calculation.
Since $S_i$ can be calculated using only the $x$-coordinate, we do not need to find $y$.
Actually, we can take $E:y^2=x^3-3x$ and a point $Q$ with the $x$-coordinate $-1$.
Let us check this.
Suppose $M_k$ is prime.
Since $M_k\equiv 3 \pmod 4$, we have $\left(\frac{3}{M_k}\right)=-\left(\frac{M_k}{3}\right)=-1$ by the quadratic reciprocity low.
Hence we can take $m=3$.
Next, since $M_k\equiv -1\pmod 8$, we have $\left(\frac{(-1)^3-3(-1)}{M_k}\right)=\left(\frac{2}{M_k}\right)=1$.
Hence $\sqrt{2}\in \F_{M_k}$.
Therefore $Q=(-1, \sqrt{2})\in E(\F_{M_k})$.

In summary, the primality test for Mersenne numbers is the following.

\textbf{Algorithm for Mersenne numbers.}

Let $p=2^k-1$, $k \geq 3$.
Let $x_0=-1$, $x_{i+1}=\frac{x_i^4+6x_i^2+9}{4(x_i^3-3x_i)}$ modulo $p$ for $i \geq 0$.
Define $S_i=x_{i-1}^3-3x_{i-1}$ modulo $p$ for $i\geq 1$.

To check the primality, do the following steps.
\begin{enumerate}
	\item Calculate $S_i$ for $i=1$, $2$, \ldots, $k-1$ .
	If $\gcd(S_i,p)>1$ for some $i$, $1\leq i \leq k-1$, then $p$ is composite.
	If $\gcd(S_i,p)=1$ for $i=1$, $2$, \ldots, $k-1$, then go to Step 2.
	\item If $S_{k}\equiv0 \pmod p$, then $p$ is prime.
	If not, $p$ is composite.
\end{enumerate}

Therefore, we get a primality test which is an analogue of the Lucas-Lehmer test.

\begin{remark}
Note that for Mersenne numbers, the algorithm concludes that $p$ is composite if and only if $\gcd(S_i,p)>1$ for some $i$, $1\leq i \leq k-1$.
Hence as mentioned above, it might find a proper divisor of $p$ as a value of $\gcd(S_i,p)$.
\end{remark}

\section{Primality test for $p=2^kn-1$ when $n$ is large.} 

Next, let us consider the case when $n$ is relatively large.
For this case, we assume $n=q$ is prime for simplicity.

\textbf{ Algorithm.}
Let $p=2^kq-1$ with $k\geq 2$ and $q$ prime.
Fix $\lambda >1$ and assume $2^k\lambda \leq \sqrt{p}$ and $\lambda\sqrt{p}>(p^{1/4}+1)^2$.

To check if $p$ is prime or not, do the following steps.
\begin{enumerate}
	\item Take $x \in \Z$ such that  $\left(  \frac{x}{p}\right)=-1$ and find $y$ such that 
	 $\left(\frac{x^3-y^2}{p}\right)=1$. 
	Let $m=(x^3-y^2)/x$ mod $p$.
	Then $Q=(x,y)$ lies on the curve $E: y^2=x^3-mx$. 
	Then the following calculation is done in $E(\Z_p)$.
	\item If $2^kQ=\infty$, then go to Step 1 and take another $y$. If $2^kQ \neq \infty$, then go to Step 3.
	\item If $q(2^kQ)=\infty$, then $p$ is prime. If not, $p$ is composite.

\end{enumerate}
\begin{theorem}\label{thm:check}
	If we reach Step 3 in the above algorithm, it determines whether or not $p$ is prime.
\end{theorem}
\begin{proof}
We have $\left(\frac{m}{p}\right)=\left(\frac{(x^3-y^2)/x}{p}\right)=\left(\frac{x}{p}\right)\left(\frac{x^3-y^2}{p}\right)=-1\cdot1=-1$.
If $p$ is prime, then by Theorem \ref{thm:group-structure} we have $\EFp \cong \Z_{2^kq}$.
Since the $x$-coordinate of $Q$ is a quadratic non-residue, the order of $Q$ is divisible by $2^k$  by Theorem \ref{thm:main}.
By Step 2, we know that $2^kQ\neq\infty$. Hence $Q$ has order $2^kq$.
So if $2^kqQ \neq \infty$, then $p$ is not prime.

Suppose we have $q(2^kQ)=\infty$ in Step 3 and $p$ is composite. 
Let $r$ be a prime divisor of $p$ such that $r \leq \sqrt{p}$.
Since $2^kQ \neq \infty$ and $q(2^kQ)=\infty$, $Q$ has order divisible by $q$.
Using the assumption on $k$, we have
\[\lambda \sqrt{p}\leq p/2^k < q \leq \#E(\F_r)\leq(\sqrt{r}+1)^2 \leq (p^{1/4}+1)^2.\]
Here the third inequality is by the Hasse's Theorem.
However, we assumed this inequality does not hold.
Hence $p$ is prime.
\end{proof}

\begin{remark}
Since we know  $Q=(x,y)$, we can use the method of successive doubling when we multiply integers.
Hence it is calculated quickly.
\end{remark}
\begin{remark}
If we cannot proceed to Step 3, then this test will not stop.
However, if $q$ is large prime, then it is likely that $Q$ has order $2^kq$.
So after doing Step 2 several times if  we could not proceed to Step 3, then it is likely $p$ is composite.
Then we need to use another test to check if it is really composite.
Or we should use this test after checking that $p$ is a probably prime by another test.
\end{remark}

There exists a similar algorithm when $n$ is not prime.
However, the number of steps in the algorithm will increase.
To see what happens, let us consider the case when $n$ is a product of two primes.
Let $n=q_1q_2$, where $q_1$, $q_2$ are (not necessarily distinct) primes.
 
 \textbf{ Algorithm.}
Let $p=2^kq_1q_2-1$, where $k\geq 2$ and $q_1$, $q_2$ are primes.
Fix $\lambda >1$ and assume $2^k\lambda \leq \sqrt{p}$ and $\lambda\sqrt{p}>(p^{1/4}+1)^2$.

To check if $p$ is prime or not, do the following steps.
\begin{enumerate}
	\item Take $x \in \Z$ such that  $\left(  \frac{x}{p}\right)=-1$ and find $y$ such that 
	 $\left(\frac{x^3-y^2}{p}\right)=1$. 
	Let $m=(x^3-y^2)/x$ mod $p$.
	Then $Q=(x,y)$ lies on the curve $E: y^2=x^3-mx$. 
	Then the following calculation is done in $E(\Z_p)$.
	\item If $2^kQ=\infty$, then go to Step 1 and take another $y$. 
	If $2^kQ \neq \infty$, then go to Step 3.
	\item If $q_1(2^kQ)\neq\infty$ and  $q_2(2^kQ)\neq \infty$, then go to Step 4.
	Otherwise, go to Step 1 and take another $y$.
	\item If $q_1q_2(2^kQ)=\infty$, then $p$ is prime. 
	If not, $p$ is composite.

\end{enumerate}
 
 The proof is almost the same as that of Theorem \ref{thm:check}.
 You can replace $q$ in the proof of  Theorem \ref{thm:check} by $q_1q_2$.
 
 \begin{remark}
 
 These tests in this section correspond to the primality tests using the factors of $p+1$.
 (See \cite{Crandall-Pomerance}).
 \end{remark}
\bibliographystyle{amsplain}

\begin{thebibliography}{1}

\bibitem{Brillhart-Lehmer-Selfridge}
John Brillhart, D.~H. Lehmer, and J.~L. Selfridge, \emph{New primality criteria
  and factorizations of {$2\sp{m}\pm 1$}}, Math. Comp. \textbf{29} (1975),
  620--647. \MR{MR0384673 (52 \#5546)}

\bibitem{Crandall-Pomerance}
Richard Crandall and Carl Pomerance, \emph{Prime numbers}, Springer-Verlag, New
  York, 2001, A computational perspective. \MR{MR1821158 (2002a:11007)}

\bibitem{Goldwasser-Kilian}
Shafi Goldwasser and Joe Kilian, \emph{Primality testing using elliptic
  curves}, J. ACM \textbf{46} (1999), no.~4, 450--472. \MR{MR1812127
  (2002e:11182)}


\bibitem{Gross}
Benedict~H. Gross, \emph{An elliptic curve test for {M}ersenne primes}, J.
  Number Theory \textbf{110} (2005), no.~1, 114--119. \MR{MR2114676
  (2005m:11007)}
  
\bibitem{Husemoller}
Dale Husem{\"o}ller, \emph{Elliptic curves}, second ed., Graduate Texts in
  Mathematics, vol. 111, Springer-Verlag, New York, 2004, With appendices by
  Otto Forster, Ruth Lawrence and Stefan Theisen. \MR{MR2024529 (2005a:11078)}

\bibitem{Riesel}
Hans Riesel, \emph{Lucasian criteria for the primality of {$N=h\cdot 2\sp{n}
  -1$}}, Math. Comp. \textbf{23} (1969), 869--875. \MR{MR0262163 (41 \#6773)}

\bibitem{Washington}
Lawrence~C. Washington, \emph{Elliptic curves}, second ed., Discrete
  Mathematics and its Applications (Boca Raton), Chapman \& Hall/CRC, Boca
  Raton, FL, 2008, Number theory and cryptography. \MR{MR2404461 (2009b:11101)}

\bibitem{Williams}
Hugh~C. Williams, \emph{\'{E}douard {L}ucas and primality testing}, Canadian
  Mathematical Society Series of Monographs and Advanced Texts, 22, John Wiley
  \& Sons Inc., New York, 1998, A Wiley-Interscience Publication. \MR{MR1632793
  (2000b:11139)}

\end{thebibliography}

\end{document}